\documentclass[12pt,A4,reqno]{amsart}

\usepackage{amsfonts}
\usepackage{mathrsfs}
\usepackage[T1]{fontenc}
\usepackage{amssymb}
\usepackage{pb-diagram}
\usepackage{amsmath,amscd}
\usepackage{upgreek}
\usepackage{bm}
\usepackage{mathabx}
\usepackage{color}
\usepackage{epsf}
\usepackage{graphicx}
\usepackage{xypic}

\theoremstyle{plain}
\newtheorem{thm}{Theorem}[section]
\newtheorem{lem}[thm]{Lemma}
\newtheorem{prop}[thm]{Proposition}

\theoremstyle{definition}
\newtheorem{defi}[thm]{Definition}
\newtheorem{rem}[thm]{Remark}
\newtheorem{exam}{Example}

\newcommand{\N}{\mathbb N}

\newcommand{\R}{\mathbb R}
\newcommand{\Z}{\mathbb Z}

\newcommand{\K}{\mathcal K}
\newcommand{\I}{\mathrm I}
\newcommand{\nn}{\vskip 0.2cm}
\newcommand{\n}{\vskip 0.1cm}

\begin{document}

\title [\ ] {On lower bounds of the sum of multigraded Betti numbers of simplicial complexes}

\author{Li Yu}
\address{Department of Mathematics, Nanjing University, Nanjing, 210093, P.R.China
  }

 \email{yuli@nju.edu.cn}




 \begin{abstract}
   We find some general
   lower bounds of the sum of certain families of multigraded Betti numbers of 
   any simplicial complex with a vertex coloring.
 \end{abstract}

\maketitle

  \section{Introduction}

  An (abstract) \emph{simplicial complex} on $[m]=\{1,\cdots, m\}$ 
  is a collection $\mathcal{K}$ of subsets $\sigma\subseteq [m]$ such that
  if $\sigma\in\mathcal{K}$, then any subset of $\sigma$ also belongs to $\mathcal{K}$.
  Here we always assume that the empty set and 
  every one-element subset 
  $\{j\}_{1\leq j \leq m}$ belongs to $\mathcal{K}$ and refer
  to $\sigma\in \mathcal{K}$ as an abstract \emph{simplex} of $\mathcal{K}$.
  A simplex $\sigma \in \mathcal{K}$ of cardinality $|\sigma|=i+1$ has \emph{dimension} $i$ and is called
  an \emph{$i$-face} of $\mathcal{K}$.  In particular, any $0$-face of $\K$ is 
  called a \emph{vertex} of $\K$. The dimension $\dim(\mathcal{K})$ of $\mathcal{K}$ is defined to be 
  the maximal dimension of all the faces of $\mathcal{K}$.   
   For any subset $\omega\subseteq [m]$, we call 
   $\K|_{\omega} = \{\sigma \in \K \,|\, \sigma\subseteq \omega \}$ the \emph{full 
       subcomplex} of $\K$ corresponding to $\omega$.\n

       Let $\mathbf{k}$ denote an arbitrary field and $\mathbf{k}[\mathbf{v}]=\mathbf{k}[\mathrm{v}_1,\cdots,\mathrm{v}_m]$ be
       the polynomial ring over $\mathbf{k}$ in the $m$ indeterminates $\mathbf{v}=\mathrm{v}_1,\cdots, \mathrm{v}_m$.
       Let $\N =\{0,1,2,\cdots\}$ be the set of nonnegative integers. 
        A \emph{monomial} in $\mathbf{k}[\mathbf{v}]$ is a product $\mathbf{v}^{\mathbf{a}} = 
      \mathrm{v}_1^{a_1} \mathrm{v}_2^{a_2} \ldots \mathrm{v}_m^{a_m}$ for a vector 
      $\mathbf{a} = (a_1, \ldots ,a_m)\in \N^m$ of nonnegative integers.\n

        A $\mathbf{k}[\mathbf{v}]$-module $\mathbf{M}$ is called $\N^m$-graded if 
        $\mathbf{M}=\bigoplus_{\mathbf{b}\in \N^m} \mathbf{M}_{\mathbf{b}}$ and 
        $\mathbf{v}^{\mathbf{a}} \mathbf{M}_{\mathbf{b}} \subseteq \mathbf{M}_{\mathbf{a}+\mathbf{b}} $.      
      Given a vector $\mathbf{a}\in \N^m$, by $\mathbf{k}[\mathbf{v}](-\mathbf{a})$ one denotes the
      free $\mathbf{k}[\mathbf{v}]$-module generated by an element in degree $\mathbf{a}$. So
        $\mathbf{k}[\mathbf{v}](-\mathbf{a})$ is isomorphic to the ideal 
        $\langle\mathbf{v}^{\mathbf{a}}\rangle \subset
        \mathbf{k}[\mathbf{v}]$ as $\N^m$-graded modules. Moreover, any free $\N^m$-graded module of 
        rank $r$ is isomorphic to the direct sum $\mathbf{k}[\mathbf{v}](-\mathbf{a}_1)\oplus
         \cdots\oplus \mathbf{k}[\mathbf{v}](-\mathbf{a}_r)$ for some vectors 
         $\mathbf{a}_1,\cdots, \mathbf{a}_r\in \N^m$.\n

      An ideal $\bm{I}\subseteq \mathbf{k}[\mathbf{v}]$ is called a \emph{monomial ideal}
       if it is generated by monomials.    
       A monomial $\mathbf{v}^{\mathbf{a}}$ is called \emph{squarefree} if every coordinate of $\mathbf{a}$
        is $0$ or $1$. A monomial ideal is called squarefree if it is generated by squarefree monomials.  \n

     Clearly, all the elements in $2^{[m]} =\{ \omega\,|\, \omega\subseteq [m]\}$ bijectively correspond to all the vectors in $\{ 0 ,1\}^m \subset \N^m$
     by mapping any $\omega \in 2^{[m]}$ to $\mathbf{a}_{\omega} \in \{ 0 ,1\}^m $, where the $j$-th coordinate
     $\mathbf{a}_{\omega}$ is $1$ whenever $j\in \omega$, and is $0$ otherwise. With this understanding,
     we write 
    \begin{equation} \label{Equ:Notation-V}
      \mathbf{v}^{\omega} = \mathbf{v}^{\mathbf{a}_{\omega}}  = \prod_{j\in \omega} \mathrm{v}_j,\ \omega\subseteq [m] \quad \text{(e.g. $\mathbf{v}^{\{ j \}} = \mathrm{v}_j$)}.
      \end{equation}
   
     The \emph{Stanley-Reisner ideal} of a simplicial complex $\mathcal{K}$ on $[m]$ 
     is the squarefree monomial ideal $\bm{I}_{\mathcal{K}} = \langle \mathbf{v}^{\tau} \,|\, 
     \tau\notin \mathcal{K} \rangle$ generated by monomials in $\mathbf{k}[\mathbf{v}]$
      corresponding to nonfaces of $\mathcal{K}$.
     The quotient ring $\mathbf{k}(\mathcal{K})=\mathbf{k}[\mathbf{v}]\slash \bm{I}_{\mathcal{K}}$
     is called the \emph{Stanley-Reisner ring} of $\mathcal{K}$. By~\cite[Sect.1.4]{MilStu05},
     $\mathbf{k}(\mathcal{K})$ is a finitely generated $\N^m$-graded $\mathbf{k}[\mathbf{v}]$-module and 
     so $\mathbf{k}(\mathcal{K})$ has an $\N^m$-graded minimal free resolution as follows:
     \begin{equation} \label{Equ:Free-Resol}
       0 \longrightarrow  F_{-h} \overset{ \phi_h}{\longrightarrow} 
       F_{-h+1} \longrightarrow \cdots F_{-1} \overset{ \phi_1}{\longrightarrow}
       F_0 \longrightarrow    \mathbf{k}(\mathcal{K}) \longrightarrow  0 
     \end{equation} 
     where each $\phi_i$ is a degree-preserving homomorphism.
     Here we numerate the terms of a free resolution by nonpositive integers in order to view it as
     a cochain complex.
     Since each $F_{-i}$ is an $\N^m$-graded free $\mathbf{k}[\mathbf{v}]$-module,
     we may write 
     $$F_{-i} = \bigoplus_{\mathbf{a}\in \N^m} \mathbf{k}[\mathbf{v}](-\mathbf{a})^{\beta^{\mathbf{k}(\mathcal{K})}_{i,\mathbf{a}}}, \ i\in \N.$$
     The integers $\beta^{\mathbf{k}(\mathcal{K})}_{i,\mathbf{a}}$ are called the 
     \emph{multigraded Betti numbers} of 
     $\mathcal{K}$ with $\mathbf{k}$-coefficients.
     Since the free resolution~\eqref{Equ:Free-Resol} is minimal, we have $\mathrm{Tor}^{\mathbf{k}[\mathbf{v}]}_i (\mathbf{k}(\mathcal{K}),\mathbf{k}) \cong F_{-i}$ and
     \begin{equation}
      \beta^{\mathbf{k}(\mathcal{K})}_{i,\mathbf{a}} = 
      \dim_{\mathbf{k}} \mathrm{Tor}^{\mathbf{k}[\mathbf{v}]}_i (\mathbf{k}(\mathcal{K}),\mathbf{k})_{\mathbf{a}},
      \, \mathbf{a}\in \N^m, i\in \N.
      \end{equation}
     By~\cite[Corollary 1.40]{MilStu05}, $\beta^{\mathbf{k}(\mathcal{K})}_{i,\mathbf{a}} =0$ 
       for all $\mathbf{a}\notin \{ 0,1\}^m$. For brevity, we define
       $$ \beta^{\mathbf{k}(\mathcal{K})}_{i,\omega} = \beta^{\mathbf{k}(\mathcal{K})}_{i,\mathbf{a}_{\omega}},\
         \omega\in 2^{[m]}. $$ 
        In addition, the Hochster's formula tells us that
       (see~\cite{Hoc77} or~\cite[Corollary 5.12]{MilStu05})
     \begin{equation} \label{Equ:Hochster}
       \beta^{\mathbf{k}(\mathcal{K})}_{i,\omega} 
       = \dim_{\mathbf{k}}  \widetilde{H}^{|\omega|-i-1}(\K|_{\omega};\mathbf{k}),\ i\in \N.
     \end{equation}
      where $\widetilde{H}^{*}(\K|_{\omega};\mathbf{k})$ is the reduced singular (or simplicial) cohomology of $\K|_{\omega}$
     with $\mathbf{k}$-coefficients. Notice that $\dim(\mathcal{K}|_{\omega}) \leq
       |\omega|-1$.
     So by the Hochster's formula, the multigraded Betti numbers of $\mathcal{K}$ over $\mathbf{k}$ are nothing but the usual
     (reduced) Betti numbers with $\mathbf{k}$-coefficients of all the full subcomplexes of $\mathcal{K}$.\n
     
   The multigraded Betti numbers of $\mathcal{K}$ are intimately related to a topological space 
     $\mathcal{Z}_{\mathcal{K}}$ called the \emph{moment-angle complex of $\mathcal{K}$}.  The construction of $\mathcal{Z}_{\mathcal{K}}$ first appeared
     in Davis-Januszkiewicz~\cite{DaJan91}.
     Let $D^2 = \{ z\in \mathbb{C}\, |\, |z|\leq 1 \}$ be 
     the unit $2$-disk in $\mathbb{C}$ and $S^1 =\partial D^2$ be the unit circle.
    By definition, 
     $\mathcal{Z}_{\mathcal{K}}$ is a subspace of the Cartisian product of $m$ copies of 
     $2$-disks:
      \begin{equation} \label{Equ:Z-K}
       \mathcal{Z}_{\mathcal{K}}  = \bigcup_{\sigma\in\mathcal{K}} \big( \prod_{j\in \sigma} D^2_{(j)} \times \prod_{j\notin \sigma} S^1_{(j)} \big) \subseteq \prod^m_{j=1} D^2_{(j)}
       \end{equation}
      where $ D^2_{(j)}$ is a copy of $D^2$ indexed by $j\in[m]$ and $S^1_{(j)} =\partial  D^2_{(j)}$. In addition, 
      we consider
      the index $j \in [m]$ in~\eqref{Equ:Z-K} to be increasing from the left to the right.
     People also use $(D^2,S^1)^{\mathcal{K}}$ to denote $\mathcal{Z}_{\mathcal{K}}$ in the literature and call it the \emph{polyhedral product} of $(D^2,S^1)$ corresponding to $\K$ (see~\cite[Sec.4.2]{BP15}).
      It is shown in~\cite{BasKov02} that
     \begin{equation} \label{Equ:BasKov}
        H^{q}(\mathcal{Z}_{\mathcal{K}};\mathbf{k}) \cong 
       \bigoplus_{\omega\subseteq [m]} \widetilde{H}^{q-|\omega|-1} (\mathcal{K}|_{\omega};\mathbf{k}), \,
        q\in \N.
     \end{equation}
     So by the Hochster's formula~\eqref{Equ:Hochster}, we have
    $
          \dim_{\mathbf{k}} H^{q}(\mathcal{Z}_{\mathcal{K}};\mathbf{k}) =
           \sum_{\omega\subseteq [m]} \beta^{\mathbf{k}(\mathcal{K})}_{2|\omega|-q,\omega}$.\n
         
     If we consider $S^1$ as a Lie group with respect to the multiplication of complex numbers,
     there is a canonical action of the torus $T^m=(S^1)^m$ on $\mathcal{Z}_{\mathcal{K}}$ where
     the $j$-th $S^1$-factor of $T^m$ acts on $D^2_{(j)}$ by multiplication of complex numbers.
      
     \nn
     
     The following inequality on the multigraded Betti numbers of $\mathcal{K}$ is obtained 
     by Cao-L\"u~\cite{CaoLu12} and Ustinovskii~\cite{Uto09} independently.
     
     \begin{thm}[Theorem 1.4 in~\cite{CaoLu12}, Theorem 3.2 in~\cite{Uto09}] \label{Thm:CaoLu-Ustinov}
     For a simplicial complex $\mathcal{K}$ on $[m]$ and any field $\mathbf{k}$,
      $\sum_{q} \dim_{\mathbf{k}} H^q(\mathcal{Z}_{\mathcal{K}};\mathbf{k}) \geq 2^{m-\dim(\mathcal{K})-1}$. 
     \end{thm}
     By~\eqref{Equ:BasKov}, Theorem~\ref{Thm:CaoLu-Ustinov} is equivalent to saying that
      the sum of all the multigraded Betti numbers of $\mathcal{K}$ is at least $2^{m-\dim(\mathcal{K})-1}$:
      \begin{equation} \label{Equ:Bound-1}
         \sum_{i} \sum_{\omega\subseteq [m]} 
         \beta^{\mathbf{k}(\mathcal{K})}_{i,\omega} \geq 2^{m-\dim(\mathcal{K})-1}.  
         \end{equation}
         By Hochster's formula, the inequality~\eqref{Equ:Bound-1} 
        is equivalent to saying that the sum of all the usual Betti numbers of all 
       the full subcomplexes of $\K$ is at least $2^{m-\dim(\mathcal{K})-1}$:
       \begin{equation}
         \sum_{\omega\subseteq [m]} \Big(  \sum_j 
       \dim_{\mathbf{k}}  \widetilde{H}^{j}(\K|_{\omega};\mathbf{k}) \Big) \geq 2^{m-\dim(\mathcal{K})-1}.
       \end{equation}
     The lower bound in~\eqref{Equ:Bound-1} is sharp since the equality holds for $\partial\Delta^{n_1}*\cdots*\partial\Delta^{n_s}$ where $\partial\Delta^{n}$ is the boundary of an $n$-simplex $\Delta^{n}$ and $*$ is the join of two spaces.\n
     
     Moreover, in~\cite{Uto12} Ustinovskii shows that the inequality~\eqref{Equ:Bound-1} can be strengthened to be the following.
    \begin{thm}[Corollary 3.6 in~\cite{Uto12}] \label{Thm:Uto12}
      For a simplicial complex $\mathcal{K}$ on $[m]$ and any field $\mathbf{k}$, 
        $$\sum_{\omega\subseteq [m]} \beta^{\mathbf{k}(\mathcal{K})}_{i,\omega} \geq \binom{m-\dim(\mathcal{K})-1}{i}\ \text{for any}\ i\geq 0.$$
     \end{thm}
     The proof of this theorem in~\cite{Uto12}
      uses a result in~\cite[Corollary 2.5]{EvanGriff88} on the dimensions
     of the Tor modules of monomial ideals in $\mathbf{k[v]}$.
\n
     
      In this paper, we obtain some lower bounds of the sum of certain families of multigraded 
     Betti numbers of $\mathcal{K}$, which generalizes Theorem~\ref{Thm:Uto12}. To state our main result,
     let us introduce some notations.\n
     
     \begin{itemize}
     \item Let $\upalpha=\{\alpha_1,\cdots, \alpha_r\}$ 
       be a \emph{partition} of $[m]$ into $r$ subsets, i.e. $\alpha_i$'s are disjoint nonempty subsets of $[m]$
       with $\alpha_1 \cup\cdots \cup \alpha_r=[m]$.
       The partition $\upalpha$ is called \emph{nondegenerate for $\mathcal{K}$} if 
           the two vertices of any $1$-simplex of $\mathcal{K}$ belong to different $\alpha_i$. For example the \emph{trivial partition} $\{ \{1\},\cdots, \{m\} \}$ of $[m]$ is nondegenerate for any simplicial complex on $[m]$. We also refer to a nondegenerate partition $\upalpha$ of $[m]$ for $\K$ as a \emph{vertex coloring} of $\K$, meaning that we put $r$ different colors to all the vertices of $\K$ so that 
    the set of vertices with the $i$-th color is $\alpha_i$ and
           any adjacent vertices in $\K$ are assigned different colors. 
            \n
       
     \item Let $[r]=\{ 1,\cdots, r\}$. One should keep in mind the difference
     between $[r]$ and the vertex set $[m]$ of $\mathcal{K}$. For a 
     subset $\mathrm{L}\subseteq [r]$, define
     $$ \omega^{\upalpha}_{\mathrm{L}} = \bigcup_{i\in \mathrm{L}} \alpha_i \subseteq [m].$$
     It is clear that $|\omega^{\upalpha}_{\mathrm{L}}| = \sum_{i\in\mathrm{L}} |\alpha_i| \geq |\mathrm{L}|$.      
   Then the partition $\upalpha$ determines $2^r$ special  full subcomplexes 
       $\{ \mathcal{K}|_{\omega^{\upalpha}_{\mathrm{L}}}\,|\, \mathrm{L}\subseteq [r] \}$ of $\mathcal{K}$ (see Example~\ref{Exam:Color-Complex}).\n

       \item For any simplex $\sigma\in \K$, let
        \[
          \I_{\upalpha}(\sigma) = \{ i\in [r] \,
                  ;\,\sigma\cap \alpha_i \neq \varnothing \}
     \subseteq [r].
     \]
    In other words, $\I_{\upalpha}(\sigma)$ encodes the set of colors 
    of all the vertices of $\sigma$ defined by $\upalpha$.  Note that
    $|\I_{\upalpha}(\sigma)| \leq |\sigma|$.
            
  \end{itemize}
  
  \begin{exam} \label{Exam:Color-Complex}
   Let $\mathcal{K}$ be a simplicial complex on the vertex set 
   $[5]=\{1,2,3,4,5\}$ with a geometrical realization
   given by Figure~\ref{Fig:Complex}. Let
   $\upalpha=\{\{1\}, \{2,4\}, \{3,5\} \}$ be
   a partition of $[5]$, which is nondegenerate for $\mathcal{K}$. The geometrical realizations of the
   eight special full subcomplexes of $\mathcal{K}$
   determined by $\upalpha$ are shown
   in Figure~\ref{Fig:Subcomplexes}.
     \begin{figure}
        \begin{equation*}
        \vcenter{
            \hbox{
                  \mbox{$\includegraphics[width=0.46\textwidth]{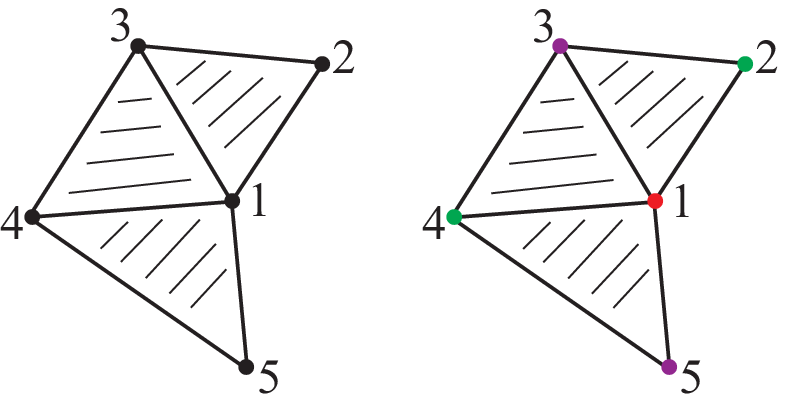}$}
                 }
           }
     \end{equation*}
   \caption{
       } \label{Fig:Complex}
   \end{figure}
    \begin{figure}
        \begin{equation*}
        \vcenter{
            \hbox{
                  \mbox{$\includegraphics[width=0.96\textwidth]{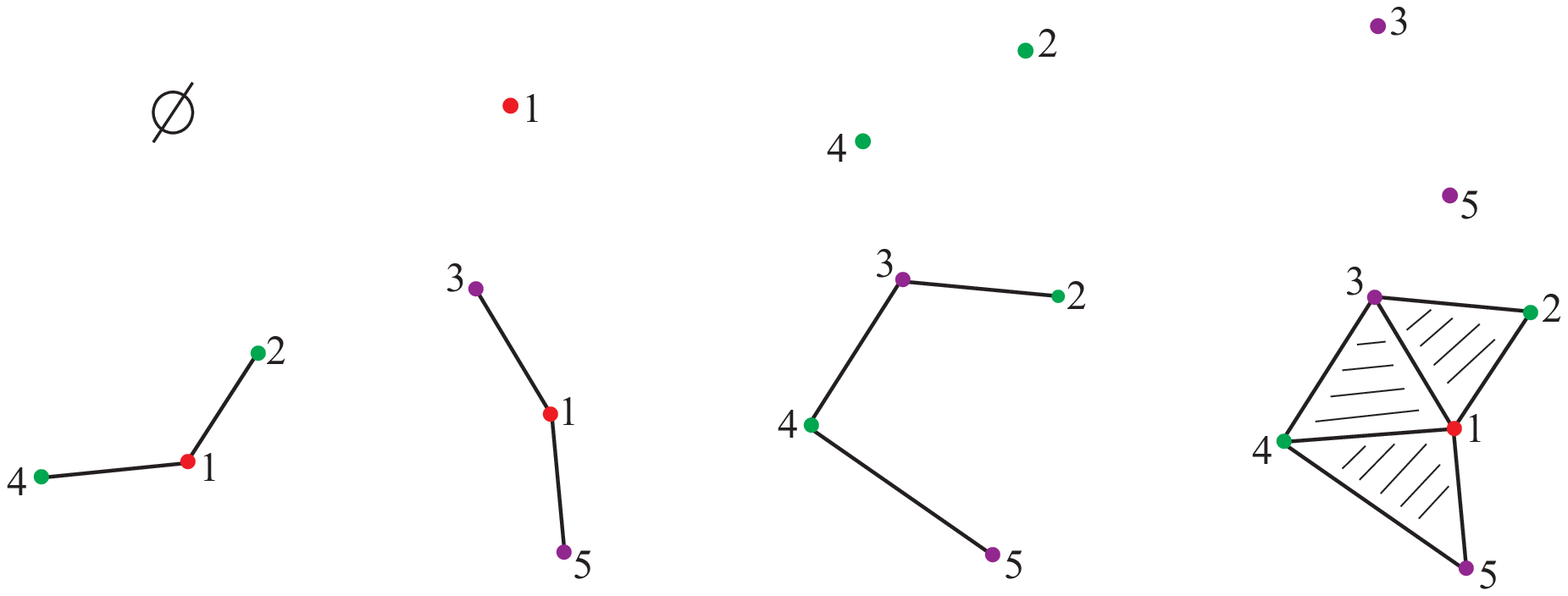}$}
                 }
           }
     \end{equation*}
   \caption{
       } \label{Fig:Subcomplexes}
   \end{figure}
   
  \end{exam}

      The main result of this paper is the following theorem.
      
      \begin{thm} \label{Thm:Main-1}
        Let $\K$ be a simplicial complex on $[m]$ and $\mathbf{k}$ be an arbitrary field.
        If a partition $\upalpha =\{\alpha_1,\cdots,\alpha_r\}$ 
          of $[m]$ is nondegenerate for $\mathcal{K}$, then for any $q\geq 0$,
        \begin{equation} \label{Equ:bound-1}
        \sum_{\mathrm{L}\subseteq [r]} \beta^{\mathbf{k}(\mathcal{K})}_{q+|\omega^{\upalpha}_{\mathrm{L}}|-|\mathrm{L}|,\omega^{\upalpha}_{\mathrm{L}}} =
        \sum_{\mathrm{L}\subseteq [r]} \dim_{\mathbf{k}} 
        \widetilde{H}^{|\mathrm{L}|-q-1}(\K|_{\omega^{\upalpha}_{\mathrm{L}}};\mathbf{k}) \geq \binom{r-\dim(\mathcal{K})-1}{q}.
         \end{equation}
           This implies
          \begin{equation} \label{Equ:bound-2}
           \sum_{\mathrm{L}\subseteq [r]} \Big(  \sum_{j=0}^{|\mathrm{L}|-1 } \dim_{\mathbf{k}} 
        \widetilde{H}^{j}(\K|_{\omega^{\upalpha}_{\mathrm{L}}};\mathbf{k}) \Big) \geq 
        2^{r-\dim(\mathcal{K})-1}.
        \end{equation}
      \end{thm}
      Note that when $r=m$ (i.e. $\upalpha$ is the trivial partition of $[m]$), Theorem~\ref{Thm:Main-1}
       gives exactly Theorem~\ref{Thm:Uto12}. For nontrivial   
      nondegenerate partitions of $[m]$ for $\K$,
       Theorem~\ref{Thm:Main-1} will give us some lower bounds of the sum of certain families of 
       multigraded Betti numbers of $\mathcal{K}$.
       The inequality~\eqref{Equ:bound-2} tells us that the
       sum of all the usual Betti numbers of the $2^r$ special
       full subcomplexes of $\mathcal{K}$ determined by $\upalpha$ also has a general lower bound.
       We will prove Theorem~\ref{Thm:Main-1} at the end of the paper.
       Note that the Hochster's formula~\eqref{Equ:Hochster} itself sheds no light on how to obtain such type of lower bounds. \n
       
       \begin{rem}
         It is shown in~\cite{Notbo10} that the existence of a
         vertex coloring by $r$ colors on a simplicial complex $\mathcal{K}$ is equivalent to some
         splitting properties of a vector bundle over the  Davis-Januszkiewicz space of $\mathcal{K}$.
       \end{rem}
       
       The paper is organized as follows.
       In Section 2, we study 
        a quotient space of $\mathcal{Z}_{\mathcal{K}}$ by some torus action determined by $\upalpha$, whose cohomology groups are related
        to the multigraded Betti numbers 
       $\beta^{\mathbf{k}(\mathcal{K})}_{i,\omega^{\upalpha}_{\mathrm{L}}}$. 
       In Section 3, we show how the multigraded Betti numbers 
       $\beta^{\mathbf{k}(\mathcal{K})}_{i,\omega^{\upalpha}_{\mathrm{L}}}$ are related to
       the Tor module of $\mathbf{k}(\mathcal{K})$ 
        over a polynomial ring $\mathbf{k}[\mathrm{u}_1,\cdots, \mathrm{u}_r]$ where the $\mathbf{k}[\mathrm{u}_1,\cdots, \mathrm{u}_r]$-module structure of
        $\mathbf{k}(\mathcal{K})$ 
        is determined by the partition $\upalpha$. This relation leads to a proof of Theorem~\ref{Thm:Main-1} at the end.
      \vskip .6cm
      
    \section{A quotient space of $\mathcal{Z}_{\mathcal{K}}$ determined by $\upalpha$}
     
      Suppose $\mathcal{K}$ is a simplicial complex on $[m]$.      
    Let $\Z^m = \langle e_1,\cdots, e_m\rangle$ be the canonical unimodular basis of
      the Lie algebra of $T^m$.
      For an arbitrary partition $\upalpha$ of $[m]$, let 
      $S_{\upalpha}$ be the toral subgroup of $T^m$ corresponding to
       the subgroup of $\Z^m$ generated by the set
       $\{ e_{j} - e_{j'} \,|\, j,j' \ \text{belong to the same}\ \alpha_i \ \text{for some}\ 
            1\leq i \leq k \}$. It is easy to see that the dimension of 
            $S_{\upalpha}$ is $m-r$.
      Let $\mathcal{Z}_{\K} \slash S_{\upalpha}$ denote the quotient space
       of $\mathcal{Z}_{\K}$ by $S_{\upalpha}$ through the 
      canonical action of $T^m$. Clearly $T^m\slash S_{\upalpha}$ is an $r$-dimensional
        torus which acts on $\mathcal{Z}_{\K} \slash S_{\upalpha}$ through
        the canonical action of $T^m$. \n
      
      \begin{lem}
      For a partition $\upalpha$ of $[m]$,
      the canonical action of $S_{\upalpha}$ on $\mathcal{Z}_{\K}$
      is free if and only if $\upalpha$ is nondegenerate for $\K$.  
       \end{lem}
      \begin{proof}
      Let $v_{j}$ be the center of the $2$-disk $D^2_{(j)}$ 
      in~\eqref{Equ:Z-K}. For any $j<j'\in [m]$, 
      let $C_{j,j'}$ denote the circle subgroup of
      $T^m$ corresponding to the subgroup of $\Z^m$ generated by $e_j - e_{j'}$. The fixed point set of the canonical action of $C_{j,j'}$
      on $D^2_{(1)}\times\cdots\times D^2_{(m)}$ is
      $ D^2_{(1)}\times \cdots \times D^2_{(j-1)}\times \{v_j\}\times D^2_{(j+1)} \times \cdots\times
     D^2_{(j'-1)}\times \{v_{j'}\}\times D^2_{(j'+1)} \times\cdots\times D^2_{(m)}$.
     In particular, the action of $C_{j,j'}$ on the component
     $\prod_{k\in \sigma} D^2_{(k)} \times \prod_{k\notin \sigma} S^1_{(k)} \subseteq \mathcal{Z}_{\K}$ has a fixed point if and only if
     $j$ and $j'$ are both in $\sigma$.
       Then it is easy to that the canonical 
       $S_{\upalpha}$-action on $\mathcal{Z}_{\K}$ is free if and only if for any simplex $\sigma\in \K$,
       $|\sigma \cap \alpha_i|\leq 1$ for all $i\in [r]$, which means that $\upalpha$ is nondegenerate for $\K$.
      \end{proof}
     
      For any partition $\upalpha$ of $[m]$, it is shown in~\cite[Theorem 1.2]{Yu2019} that there is a group isomorphism
     \begin{equation} \label{Equ:Isom-1}
        H^q(\mathcal{Z}_{\K} \slash S_{\upalpha};\mathbf{k}) \cong 
          \underset{\mathrm{L}\subseteq [r]}{\bigoplus} \widetilde{H}^{q-|\mathrm{L}|-1} 
          (\K|_{\omega^{\upalpha}_{\mathrm{L}}};\mathbf{k}), \, \forall q\geq 0.
        \end{equation}   
  So by the Hochster's formula~\eqref{Equ:Hochster}, we have
   \begin{equation} \label{Equ:dfd}
     \dim_{\mathbf{k}}  H^q(\mathcal{Z}_{\K} \slash S_{\upalpha};\mathbf{k}) = \sum_{\mathrm{L}\subseteq [r]} \beta^{\mathbf{k}(\mathcal{K})}_{|\omega^{\upalpha}_{\mathrm{L}}|+|\mathrm{L}|-q,
   \omega^{\upalpha}_{\mathrm{L}} }.  
   \end{equation}
     If the partition $\upalpha$ of $[m]$ is nondegenerate for $\K$, 
     the canonical action of $S_{\upalpha}$ on $\mathcal{Z}_{\K}$
      is free. By~\cite[Theorem 1]{Panov15} (also see~\cite[7.37]{BP02}) there is an isomorphism of graded algebras for the free quotient $\mathcal{Z}_{\mathcal{K}}\slash S_{\upalpha}$:
      \begin{equation} \label{Equ:Isom-2}
        H^*(\mathcal{Z}_{\mathcal{K}}\slash S_{\upalpha};\mathbf{k}) \cong 
        \mathrm{Tor}^{H^*(B(T^m\slash S_{\upalpha});\mathbf{k})}(\mathbf{k}(\mathcal{K}),\mathbf{k}) 
        \end{equation} 
     where $B(T^m\slash S_{\upalpha})$ is the classifying space for the $r$-dimensional torus $T^m\slash S_{\upalpha}$.
 By combining the isomorphisms in~\eqref{Equ:Isom-1} and~\eqref{Equ:Isom-2}, we
  could compute the sum of the multigraded Betti numbers 
  in~\eqref{Equ:dfd} by
   $ \mathrm{Tor}^{H^*(B(T^m\slash S_{\upalpha});\mathbf{k})}(\mathbf{k}(\mathcal{K}),\mathbf{k})$.
   But the proof in~\cite{Panov15} does not give us any explicit formula for the isomorphism in~\eqref{Equ:Isom-2}, which makes this computation a little vague. In addition,
   the sum of multigraded Betti numbers in~\eqref{Equ:dfd} is a little different from
   the sum appearing in Theorem~\ref{Thm:Main-1}.
    In Section 3, we will show 
  how each 
  $\beta^{\mathbf{k}(\mathcal{K})}_{i,\omega^{\upalpha}_{\mathrm{L}}}$
   is related to
     $\mathrm{Tor}^{H^*(B(T^m\slash S_{\upalpha});\mathbf{k})}(\mathbf{k}(\mathcal{K}),\mathbf{k})$ via a cell decomposition of $\mathcal{Z}_{\K} \slash S_{\upalpha}$ defined below. This cell decomposition
     of $\mathcal{Z}_{\K} \slash S_{\upalpha}$ is originally
   constructed in~\cite[Section 2]{Yu2019}.\n
    
     Let 
     $\Delta^{[m]}$ be a $(m-1)$-dimensional geometric simplex, which is 
     the convex hull of $m$ points $v_1,\cdots,v_m$ in general position in the
     Euclidean space $\R^m$. Then
     $\mathcal{K}$ is geometrically realized as a subset of $\Delta^{[m]}$,
     where any simplex $\sigma=\{j_1,\cdots,j_s\}\in\K$ is realized as
     the convex hull of $v_{j_1},\cdots, v_{j_s}$, or equivalently the join
     $v_{j_1}*\cdots*v_{j_s}$.
     For any $\omega\subseteq [m]$, let
     $\Delta^{\omega}$ denote the face of $\Delta^{[m]}$ whose 
     vertex set is $\omega$, i.e.
      $$\Delta^{\omega}= \underset{j\in \omega}{\Asterisk} v_j.$$
    
   If we think of $D^2$ as $S^1*v$,
       we can rewrite the decomposition of $\mathcal{Z}_{\K}$ 
      in~\eqref{Equ:Z-K} as:
       \begin{equation}\label{Equ:Construction-Complex-Moment_3}
     \mathcal{Z}_{\K}= 
        \bigcup_{\sigma\in \mathcal{K}} \Big(   
        \prod_{j\in \sigma} \big( S^1_{(j)}*v_j \big)\times 
           \prod_{j\notin \sigma} S^1_{(j)}\Big).
    \end{equation}  

  From this decomposition of $ \mathcal{Z}_{\K}$,
    we can obtain a decomposition of $\mathcal{Z}_{\K} \slash S_{\upalpha}$
    as follows (see~\cite[Section 2]{Yu2019}):
       \begin{equation}\label{Equ:Decomp-Quotient}
       \mathcal{Z}_{\K} \slash S_{\upalpha} 
          = \bigcup_{\sigma\in \K} \Big(
       \prod_{i\in \mathrm{I}_{\upalpha}(\sigma)} 
       \big(S^1_{\{ i \}} * \underset{j\in \sigma\cap \alpha_i}{\Asterisk} v_j \big)\times
        \prod_{i\in [r]\backslash \mathrm{I}_{\upalpha}(\sigma)} S^1_{\{i\}} \Big)  \subseteq \prod_{i\in [r]} S^1_{\{i\}} * \Delta^{\alpha_i}
       \end{equation} 
    where $S^1_{\{i\}} = e^1_{\{i\}}\cup e^0_{\{i\}}$ denotes a copy of $S^1$ corresponding to $i\in [r]$ and, we consider the index
    $i\in [r]$ to be increasing from the left to the right. \n
    \noindent  \textbf{Warning:}
   We use the subscript $_{(j)}$ for $j\in [m]$  while use subscript $_{\{i\}}$ for $i\in [r]$. 
  \n
     A natural cell decomposition of 
  $S^1_{\{i\}} * \Delta^{\alpha_i}$ is given by
  $$\{ (S^1 * \Delta^{\tau})^{\circ}\,|\, \tau \subseteq \alpha_i, \tau\neq \varnothing \}\cup \{ e^1_{\{i\}}, e^0_{\{i\}} \} $$
  where $(S^1 * \Delta^{\tau})^{\circ}$ is the interior of 
  $S^1*\Delta^{\tau}$. If we consider these cells as the basis of the
  cellular cochain complex of $S^1_{\{i\}} * \Delta^{\alpha_i}$, then
  \begin{equation} \label{Equ:Cobound-Simplex}
  \bullet\ \text{the coboundary of} \ 
   (S^1 * \Delta^{\tau})^{\circ} =
   \underset{|\omega|=|\tau|+1}{\sum_{\tau\subset\omega\subseteq \alpha_i}} \pm
  (S^1 * \Delta^{\omega})^{\circ},\ \tau\neq \varnothing. \quad
  \end{equation}
  \begin{itemize}
    \item the coboundary of
   $e^1_{\{i\}}$ is $\sum_{j\in \alpha_i}(S^1 * v_{j})^{\circ}$, and the coboundary of $e^0_{\{i\}}$ is zero. 
   \end{itemize}
   \n

      In the rest of this section, we always assume that
    the partition $\upalpha$ of $[m]$ is nondegenerate for $\mathcal{K}$. 
   Then for any $\sigma\in\K$, we have
   $|\mathrm{I}_{\upalpha}(\sigma)|=|\sigma|$. In other words,
    $\sigma\cap \alpha_i$ is a single vertex of $\sigma$ for any $i\in \mathrm{I}_{\upalpha}(\sigma)$.
    For convenience, we define
      $$\sigma_{\{i\}} = \sigma\cap \alpha_i \in [m], \, i\in \mathrm{I}_{\upalpha}(\sigma).$$
      By this notation, $\sigma = \bigcup_{i\in \mathrm{I}_{\upalpha}(\sigma)} \sigma_{\{i\}} \subseteq [m]$.
      So we can rewrite~\eqref{Equ:Decomp-Quotient} as
     \begin{align} \label{Equ:Decomp-Quotient-2}
      \mathcal{Z}_{\K} \slash S_{\upalpha} 
         & = \bigcup_{\sigma\in \K} \Big(
       \prod_{i\in \mathrm{I}_{\upalpha}(\sigma)} 
       \big(S^1_{\{i\}} *  v_{\sigma_{\{i\}}} \big)\times
        \prod_{i\in [r]\backslash \mathrm{I}_{\upalpha}(\sigma)} S^1_{\{i\}} \Big) .
     \end{align} 
  
      For any simplex $\sigma\in\K$ and $\mathrm{I}\subseteq [r]\backslash
     \mathrm{I}_{\upalpha}(\sigma) $, define 
   \begin{align}\label{Equ:Cell-1} 
    B_{(\sigma,\mathrm{I})} &=  \prod_{i\in \mathrm{I}_{\upalpha}(\sigma)} 
       \big(S^1_{\{i\}} *  v_{\sigma_{\{i\}}} \big)^{\circ} \times \prod_{i\in \mathrm{I}} e^{1}_{\{i\}} \times
        \prod_{i\in [r]\backslash (\mathrm{I}_{\upalpha}(\sigma))\cup\mathrm{I})} e^0_{\{i\}} \ \subset \mathcal{Z}_{\K} \slash S_{\upalpha}.
    \end{align}
   Then 
   $B_{(\sigma,\mathrm{I})}$ is a cell of dimension $2|\I_{\upalpha}(\sigma)|+|\mathrm{I}| = 2|\sigma|+|\mathrm{I}|$.  
   It is easy to see that 
   $$\{ B_{(\sigma,\mathrm{I})} \, |\, \sigma \in \K,\,
     \mathrm{I} \subseteq [r]\backslash \mathrm{I}_{\upalpha}(\sigma)\}$$ 
     is a cell decomposition of $\mathcal{Z}_{\K} \slash S_{\upalpha}$.
     Let $C^{*}(\mathcal{Z}_{\K} \slash S_{\upalpha};\mathbf{k})$ denote the cellular cochain complex
     determined by this cell decomposition whose generators are denoted by 
     $$\{ B^*_{(\sigma,\mathrm{I})} \, |\, \sigma \in \K,\,
     \mathrm{I} \subseteq [r]\backslash \mathrm{I}_{\upalpha}(\sigma)\}.$$ 
  To write the coboundary map of $C^{*}(\mathcal{Z}_{\K} \slash S_{\upalpha};\mathbf{k})$, 
    we introduce the following conventions and notations.
    \begin{itemize}
         
  \item We orient the cell $B_{(\sigma,\mathrm{I})}$
        by ordering the factors in~\eqref{Equ:Cell-1} so that the index $i \in [r]$ is
        increasing from the left to the right
        and,
    orienting the circle $S^1_{\{i\}}$ and the cells
   $e^1_{\{i\}}$ and $e^0_{\{i\}}$ in the same way for all different $i\in [r]$. \n
   
    \item  For any $i\in \mathrm{L} \subseteq [r]$, define
     $\kappa(i,\mathrm{L})= (-1)^{c(i,\mathrm{L})}$, where
     $c(i,\mathrm{L})$ is the number of elements in $\mathrm{L}$ that are less than $i$. 
     For any $\mathrm{L}, \mathrm{L'}\subseteq [r]$ with $\mathrm{L}\cap \mathrm{L'}=\varnothing$,
     we clearly have  $\kappa(i,\mathrm{L}\cup\mathrm{L'}) = \kappa(i,\mathrm{L})\cdot \kappa(i,\mathrm{L}')$. 
   \end{itemize} 
 \n
 
  By the above conventions,  the coboundary of 
 $B^*_{(\sigma,\mathrm{I})}$ is given by:
      \begin{align} \label{Equ:Boundary-1}
          d B^*_{(\sigma,\mathrm{I})} & =  \sum_{i\in \mathrm{I}} 
            \kappa(i, \mathrm{I} )\, \Big(
          \underset{|\omega| = |\sigma|+1}
          {\sum_{\sigma\subset\omega\in\K,\, \omega\backslash \sigma \in \alpha_i}}
          B^*_{(\omega,\mathrm{I}\backslash \{i\})}  \Big) 
     \end{align} 
     where $\omega\backslash \sigma \in [m]$ denotes the only vertex of 
      $\omega\in \K$ that is not contained in $\sigma$. 
      By~\eqref{Equ:Cobound-Simplex}, the coboundary of any factor $\big(S^1_{\{i\}} *  v_{\sigma_{\{i\}}} \big)^{\circ}$ in $B_{(\sigma,\mathrm{I})}$ is zero since there is no simplex $\omega \in \K$ 
      with $\omega\cap \alpha_{i}$ consisting of more than one vertex. In addition, when taking the coboundary of 
       $B^*_{(\sigma,\mathrm{I})}$ via the Leibniz' rule, passing an even dimensional factor has no contribution to the sign.
    \n               
 
     For any $\mathrm{L}\subseteq [r]$, denote by 
      $C^{*,\mathrm{L}}(\mathcal{Z}_{\K} \slash S_{\upalpha};\mathbf{k})$  
      the $\mathbf{k}$-submodule
       of $C^*(\mathcal{Z}_{\K} \slash S_{\upalpha};\mathbf{k})$ generated by the set   
          $\{ B^*_{(\sigma, \mathrm{I})} \,|\, \sigma\in\K, \,\mathrm{I}_{\upalpha}(\sigma)\cup 
            \mathrm{I} =\mathrm{L},\, \mathrm{I}_{\upalpha}(\sigma)\cap 
            \mathrm{I} =\varnothing \}$.
   By the formula~\eqref{Equ:Boundary-1}, it is easy to see that
   $C^{*,\mathrm{L}}(\mathcal{Z}_{\K} \slash S_{\upalpha};\mathbf{k})$ is a cochain subcomplex
   of $C^*(\mathcal{Z}_{\K} \slash S_{\upalpha};\mathbf{k})$. So we have a decomposition of cochain complexes: 
    \[ C^*(\mathcal{Z}_{\K} \slash S_{\upalpha};\mathbf{k})  = \bigoplus_{\mathrm{L}\subseteq [r]} 
    C^{*,\mathrm{L}}(\mathcal{Z}_{\K} \slash S_{\upalpha};\mathbf{k}) . \]
    
Next, we introduce another space $\mathcal{X}(\K, \upalpha)$ which is
   homotopy equivalent to $\mathcal{Z}_{\K} \slash S_{\upalpha}$.
   We will see in Section 3 that $\mathcal{X}(\K, \upalpha)$ plays an important role in relating the cohomology of $\mathcal{Z}_{\K} \slash S_{\upalpha}$ with
   $\mathrm{Tor}^{H^*(B(T^m\slash S_{\upalpha});\mathbf{k})}(\mathbf{k}(\mathcal{K}),\mathbf{k})$. 
   Let 
   $S^{\infty}_{(j)}$ denote a copy of infinite dimensional sphere $S^{\infty}$ based at $v_j$ for each
   $j\in[m]$. We define
       \begin{equation} \label{Equ:Decomp-Complex-Deform}
      \mathcal{X}(\K, \upalpha) := \bigcup_{\sigma\in \K} \Big(
       \prod_{i\in \mathrm{I}_{\upalpha}(\sigma)} 
       \big(S^1_{\{i\}} * S^{\infty}_{(\sigma_{\{i\}})} \big)\times
        \prod_{i\in [r]\backslash \mathrm{I}_{\upalpha}(\sigma)} S^1_{\{i\}} \Big)
        \subseteq  \prod_{i\in [r]} \big( S^1_{\{i\}} *  \underset{j\in \alpha_i}{\Asterisk}
         S^{\infty}_{(j)}\big).
   \end{equation} 
   So $\mathcal{Z}_{\K} \slash S_{\upalpha}$ is a subspace of $\mathcal{X}(\K, \upalpha)$. By   
   contracting each $S^{\infty}_{(j)}$ to $v_{j}$, we obtain
   a deformation contraction from $ \mathcal{X}(\K, \upalpha)$ to 
   $\mathcal{Z}_{\K} \slash S_{\upalpha}$.
    The definition of $\mathcal{X}(\K, \upalpha)$
   is inspired by the deformation of
   $\mathcal{Z}_{\K} = (D^2,S^1)^{\K}$ to the polyhedral product 
   $(S^{\infty}, S^1)^{\mathcal{K}}$ introduced in~\cite[Sec.4.5]{BP15}.
    \n

   Next, we construct
    a cell decomposition of $\mathcal{X}(\K, \upalpha)$. There is a natural cell decomposition on $S^{\infty}$ 
   with exactly one cell $\xi^n$ in each dimension $n\geq 0$, where 
   the boundary of $\xi^{2k}$ 
   is the closure of $\xi^{2k+1}$ and the boundary of 
   $\xi^{2k+1}$ is zero for every $k\geq 0$. 
    We denote the cells in $S^{\infty}_{(j)}$ 
    by $\{ \xi^n_{(j)} \}_{n \geq 0}$ for any $j\in [m]$.
    In particular, $\xi^0_{(j)} = v_j$
    for any $j\in [m]$.\n

    \begin{defi}[Weight]
   We call any function $\mathbf{h}: [m]\rightarrow \N$ a \emph{weight} on $[m]$.
   The \emph{support} of $\mathbf{h}$ is defined to be
    supp$(\mathbf{h})=\{ j\in [m]\,|\, \mathbf{h}(j)\neq 0 \}$.
   \begin{itemize}
   \item  For any $j\in [m]$, define 
   $\bm{\delta}_j : [m]\rightarrow \N$ by
      $\bm{\delta}_j(j') =0$ if $j'\neq j$ and $\bm{\delta}_j(j)=1$.\n
      
      \item For any $\sigma \subseteq [m]$, let $\mathbf{1}_{\sigma} =
       \sum_{j\in\sigma} \bm{\delta}_j$ whose support is $\sigma$.
      \end{itemize}

   \end{defi}

   For any $\sigma\in\K$, $\mathrm{I}\subseteq [r]$
   and any weight
    $\mathbf{h}$ on $[m]$ with supp$(\mathbf{h})=\sigma$, define a cell 
     $\mathbf{B}_{(\sigma,\mathbf{h},\mathrm{I})} \subset
     \mathcal{X}(\K, \upalpha)$ by
   \begin{align}\label{Equ:Cell-2} 
   \mathbf{B}_{(\sigma,\mathbf{h},\mathrm{I})} &= \prod_{i\in \mathrm{I}_{\upalpha}(\sigma)\backslash \mathrm{I}}   
            \Big(S^1_{\{i\}} * \xi^{2\mathbf{h}(\sigma_{\{i\}})-2}_{(\sigma_{\{i\}})} 
            \Big)^{\circ} \times
             \prod_{i\in \mathrm{I}\cap\mathrm{I}_{\upalpha}(\sigma)} 
       \Big( S^1_{\{i\}} * \xi^{2\mathbf{h}(\sigma_{\{i\}})-1}_{(\sigma_{\{i\}})}
       \Big)^{\circ} \\
        &\ \ \times \prod_{i\in \mathrm{I}  \backslash \mathrm{I}_{\upalpha}(\sigma)} e^{1}_{\{i\}}   \times
        \prod_{i\in [r]\backslash (\mathrm{I}_{\upalpha}(\sigma)\cup \mathrm{I})} e^0_{\{i\}}. \notag
        \end{align}    
 It is easy to see that  $\{ \mathbf{B}_{(\sigma, \mathbf{h}, \mathrm{I})} \, |\, \sigma \in \K,\, \text{supp}(\mathbf{h})=\sigma,\, \mathrm{I} \subseteq [r]\}$ 
     is
     a cell decomposition of $\mathcal{X}(\K, \upalpha)$.
     Note that $\dim \mathbf{B}_{(\sigma,\mathbf{h},\mathrm{I})}=
  |\mathrm{I}| + \sum_{j\in\sigma} 2\mathbf{h}(j)$.\n
    
   Let $\iota_{\upalpha} : \mathcal{Z}_{\mathcal{K}}\slash S_{\upalpha} \hookrightarrow 
      \mathcal{X}(\mathcal{K},\upalpha)$ be the inclusion map. So we have
      $$\iota_{\upalpha} ( B_{(\sigma,\mathrm{I})}) =\mathbf{B}_{(\sigma,\mathbf{1}_{\sigma},\mathrm{I})},\
      \sigma\in \K,\, \mathrm{I}\cap
     \mathrm{I}_{\upalpha}(\sigma) =\varnothing.$$
     
      Let $C^{*}(\mathcal{X}(\K, \upalpha);\mathbf{k})$ be the cellular cochain complex
     determined by this 
     cell decomposition of $\mathcal{X}(\K, \upalpha)$, whose generators are denoted by 
     $$\{ \mathbf{B}^*_{(\sigma, \mathbf{h}, \mathrm{I})} \, |\, \sigma \in \K,\, 
     \text{supp}(\mathbf{h})=\sigma,\,    \mathrm{I} \subseteq [r]\}.$$
     
   Similarly to $B_{(\sigma,\mathrm{I})}$, we orient the cell $\mathbf{B}_{(\sigma, \mathbf{h}, \mathrm{I})}$
    by ordering the factors in~\eqref{Equ:Cell-2}
     so that the index $i\in [r]$   
     is increasing from the left to the right and,
     orienting the cells in all different copies of $S^1_{\{i\}}$ and $S^{\infty}_{\{i\}}$ in the same way. Then it is easy to see that the coboundary of $\mathbf{B}^*_{(\sigma, \mathbf{h}, \mathrm{I})}$ is given by:
 \begin{align} \label{Equ:Boundary-2}
    d\mathbf{B}^*_{(\sigma, \mathbf{h}, \mathrm{I})} & = 
    \sum_{i\in  \mathrm{I} \cap \mathrm{I}_{\upalpha}(\sigma)} 
   \kappa\big( i, \mathrm{I} \big)\cdot  \mathbf{B}^*_{(\sigma, \mathbf{h}+\bm{\delta}_{\sigma_{\{i\}}}, \mathrm{I}\backslash\{i\})}  \\
   &+ 
  \sum_{i\in \mathrm{I}\backslash \mathrm{I}_{\upalpha}(\sigma)} 
  \underset{|\omega|=|\sigma|+1} 
            {\sum_{\sigma\subset \omega\in\K,\, \omega\backslash \sigma \in \alpha_i}} 
             \kappa\big( i, \mathrm{I} \big) 
             \cdot
   \mathbf{B}^*_{(\omega,\mathbf{h}+\bm{\delta}_{\omega\backslash\sigma},\mathrm{I}\backslash\{ i\})}. \notag
 \end{align}
       
 For any $\mathrm{L}\subseteq [r]$, denote by 
      $C^{*,\mathrm{L}}(\mathcal{X}(\K, \upalpha);\mathbf{k})$  
      the $\mathbf{k}$-submodule
       of $C^*(\mathcal{X}(\K, \upalpha);\mathbf{k})$ generated by the set   
          $\{ \mathbf{B}^*_{(\sigma,\mathbf{h}, \mathrm{I})} \,|\,
           \sigma\in\K, \, \text{supp}(\mathbf{h})=\sigma,\,
      \mathrm{I}_{\upalpha}(\sigma)\cup 
            \mathrm{I} =\mathrm{L}
       \}$.
   By~\eqref{Equ:Boundary-2}, it is easy to see that
   $C^{*,\mathrm{L}}(\mathcal{X}(\K, \upalpha);\mathbf{k})$
    is a cochain subcomplex
   of $C^{*}(\mathcal{X}(\K, \upalpha);\mathbf{k})$. So we have a decomposition of cochain complexes: 
    \[ C^{*}(\mathcal{X}(\K, \upalpha);\mathbf{k}) = \bigoplus_{\mathrm{L}\subseteq [r]} 
  C^{*,\mathrm{L}}(\mathcal{X}(\K, \upalpha);\mathbf{k}) . \]

   \vskip .6cm

     \section{The multigraded structure of $\mathrm{Tor}^{H^*(B(T^m\slash S_{\upalpha});\mathbf{k})}(\mathbf{k}(\mathcal{K}),\mathbf{k})$}
     
      Let $\upalpha = \{ \alpha_1,\cdots, \alpha_r\}$ be 
       partition of $[m]$ that is nondegenerate for $\K$. 
       Let us see how to compute $\mathrm{Tor}^{H^*(B(T^m\slash S_{\upalpha});\mathbf{k})}(\mathbf{k}(\mathcal{K}),\mathbf{k})$. Note that $H^*(B(T^m\slash S_{\upalpha});\mathbf{k})$ can be identified with the polynomial ring
       $\mathbf{k}[\mathrm{u}_1,\cdots,\mathrm{u}_r]$ with $r$ variables.      
     According to~\cite[7.37]{BP02} (also see~\cite[Sec.4.8]{BP15}), the $\mathbf{k}[\mathrm{u}_1,\cdots,\mathrm{u}_r]$-module structure
     on $\mathbf{k}(\mathcal{K})$ is given by 
     \begin{align} \label{Equ:Mod-Struc}
        \mathbf{k}[\mathrm{u}_1,\cdots,\mathrm{u}_r] & \longrightarrow \mathbf{k}(\mathcal{K}) \\
                       \mathrm{u}_i \ & \longmapsto \sum_{j\in \alpha_i} \mathrm{v}_j, \ 1\leq i\leq r.\notag
     \end{align}
      Let $\Lambda_{\mathbf{k}}[t_1,\cdots, t_r]$ denote the exterior algebra
      over $\mathbf{k}$ with $r$ generators $t_1,\cdots, t_r$.
    We define a differential $d_{\upalpha}$ on the tensor product
     $\Lambda_{\mathbf{k}}[t_1,\cdots, t_r] \otimes \mathbf{k}(\K)$ 
      by:
     \begin{equation}\label{Equ:R-boundary-1}
        d_{\upalpha}(t_i) = \sum_{j\in \alpha_i} \mathrm{v}_j ,\ 1\leq i \leq r,\
           \ \text{and}
      \ \ d_{\upalpha}(\mathrm{v}_j)=0,\ 1\leq j \leq m. 
     \end{equation}                 
   We can
  compute $\mathrm{Tor}^{\mathbf{k}[\mathrm{u}_1,\cdots,\mathrm{u}_r]}(\mathbf{k}(\K),\mathbf{k})
  \cong \mathrm{Tor}^{\mathbf{k}[\mathrm{u}_1,\cdots, \mathrm{u}_r]}(\mathbf{k}, \mathbf{k}(\K))$ 
  via the Koszul resolution of $\mathbf{k}$ (see~\cite[Sec 3.4]{BP02}).
  Recall that    
     the \emph{Koszul resolution} of $\mathbf{k}$ is:
 \[ 0\longrightarrow  \Lambda^r_{\mathbf{k}}[t_1,\cdots, t_r]\otimes\mathbf{k}[\mathrm{u}_1,\cdots, \mathrm{u}_r]
 \longrightarrow \cdots \qquad\qquad\qquad\qquad\qquad\qquad
  \]
  \[\qquad\qquad\qquad \ \longrightarrow \Lambda^1_{\mathbf{k}}[t_1,\cdots, t_r]\otimes\mathbf{k}[\mathrm{u}_1,\cdots, \mathrm{u}_r]
  \longrightarrow \mathbf{k}[\mathrm{u}_1,\cdots, \mathrm{u}_r] \longrightarrow \mathbf{k} \longrightarrow 0, \]
  where $d (t_i)=\mathrm{u}_i$ and $d (\mathrm{u}_i)=0$ for any $1\leq i \leq r$.
  Then we obtain
   \begin{align*}
   \mathrm{Tor}^{\mathbf{k}[\mathrm{u}_1,\cdots, \mathrm{u}_r]}(\mathbf{k}, \mathbf{k}(\K))
     &\cong 
       H^*\big( \big(\Lambda_{\mathbf{k}}[t_1,\cdots, t_r]\otimes\mathbf{k}[\mathrm{u}_1,\cdots, \mathrm{u}_r]\big)
      \otimes_{\mathbf{k}[\mathrm{u}_1,\cdots, \mathrm{u}_r]}  \mathbf{k}(\K) \big)\\
     & \cong H^*\big(\Lambda_{\mathbf{k}}[t_1,\cdots, t_r]\otimes \mathbf{k}(\K) \big).
   \end{align*}
  It is easy to see that the
    differential on $\Lambda_{\mathbf{k}}[t_1,\cdots, t_r]\otimes \mathbf{k}(\K)$  
  induced from 
  $\mathrm{Tor}^{\mathbf{k}[\mathrm{u}_1,\cdots, \mathrm{u}_r]}(\mathbf{k}, \mathbf{k}(\K))$ 
  is exactly given by $d_{\upalpha}$ in~\eqref{Equ:R-boundary-1}.
  \n
  By the definition of $\mathbf{k}[\K]$,     
      any element of $\mathbf{k}[\K]$
  can be uniquely written as a linear combination of monomials of the form $\mathrm{v}_{j_1}^{a_1} \ldots \mathrm{v}_{j_s}^{a_s}$ where 
  $\{j_1,  \ldots j_s\} =\sigma$ is a simplex in $\K$ and 
  $a_l\geq 1$ for all $l=1,\cdots, s$. 
  For brevity, we denote
  the monomial $\mathrm{v}_{j_1}^{a_1} \ldots \mathrm{v}_{j_s}^{a_s}$ by $\mathbf{v}^{(\sigma,\mathbf{h})}$ where $\mathbf{h}: [m]\rightarrow \N$ is a weight 
  with supp$(\mathbf{h}) = \sigma$ defined by: 
  $$\mathbf{h}(j) = 
     \begin{cases}
     a_l ,  &  \text{if $j=j_l \in \sigma$  }; \\
      0 ,  &  \text{if $j\notin \sigma$}.
 \end{cases}
 $$ 
 Comparing this notation with~\eqref{Equ:Notation-V}, we see that
 $ \mathbf{v}^{\sigma} = \mathbf{v}^{(\sigma,\mathbf{1}_{\sigma})}$ for any
 $\sigma\in\K$.
 It is clear that $\Lambda_{\mathbf{k}}[t_1,\cdots, t_r]\otimes \mathbf{k}(\K)$ is generated, as a $\mathbf{k}$-module, 
     by the set
    $$\{ t_{\mathrm{I}}\mathbf{v}^{(\sigma,\mathbf{h})}\,|\,
  \sigma\in\K,\, \text{supp}(\mathbf{h})=\sigma,\, \mathrm{I}\subseteq [r]\},$$ 
   where $t_{\mathrm{I}} = t_{i_1}\cdots t_{i_s}$,
        $\mathrm{I}=\{ i_1,\cdots, i_s\} \subseteq [r]$ with $i_1 < \cdots < i_s$. 
\n
  
  In addition, we define a multigrading $\mathrm{mdeg}^{\upalpha}$ on $\Lambda_{\mathbf{k}}[t_1,\cdots, t_r]\otimes \mathbf{k}(\K)$ by:
 \begin{align*}
    \mathrm{mdeg}^{\upalpha}(t_i) &= (-1, \{i\}) \in (\Z, [r]),\  i \in [r];\\
   \mathrm{mdeg}^{\upalpha}(\mathbf{v}^{\sigma})&=
       (0 ,\mathrm{I}_{\upalpha}(\sigma)) \in (\Z,[r]),\ \sigma\in\K.
  \end{align*}
 So $\mathrm{mdeg}^{\upalpha}(t_{\mathrm{I}}\mathbf{v}^{(\sigma,\mathbf{h})}) =(-|\mathrm{I}|,  \mathrm{I}_{\upalpha}(\sigma)\cup \mathrm{I})$. By the definition
 of $d_{\upalpha}$ in~\eqref{Equ:R-boundary-1}, we have
 \begin{align} \label{Equ:Boundary-3}
  d_{\upalpha}(t_{\mathrm{I}}\mathbf{v}^{(\sigma,\mathbf{h})}) &=
   \sum_{i\in \mathrm{I}} 
           \kappa(i,\mathrm{I})\cdot t_{\mathrm{I}\backslash\{i\}} 
          \Big( \sum_{j\in \alpha_i} \mathrm{v}_{j} 
           \mathbf{v}^{(\sigma,\mathbf{h})} \Big) = \sum_{i\in \mathrm{I}\cap\mathrm{I}_{\upalpha}(\sigma)} \kappa(i,\mathrm{I})\cdot t_{\mathrm{I}\backslash\{i\}} \mathbf{v}^{(\sigma,\mathbf{h}+ \bm{\delta}_{\sigma_{\{i\}}})} \notag\\
  & \qquad\qquad\qquad\qquad +  \sum_{i\in \mathrm{I}\backslash \mathrm{I}_{\upalpha}(\sigma)} 
  \underset{|\omega|=|\sigma|+1} 
            {\sum_{\sigma\subset \omega\in\K,\, \omega\backslash \sigma \in \alpha_i}}   
  \kappa(i,\mathrm{I}) \cdot t_{\mathrm{I}\backslash\{i\}} \mathbf{v}^{(\omega,\mathbf{h}+\bm{\delta}_{\omega\backslash \sigma})}.  
 \end{align}  
 For any subset $\mathrm{L}\subseteq [r]$,
define $(\Lambda_{\mathbf{k}}[t_1,\cdots, t_r]\otimes \mathbf{k}(\K))^{\mathrm{L}}$ to be the 
  $\mathbf{k}$-submodule of $\Lambda_{\mathbf{k}}[t_1,\cdots, t_r]\otimes \mathbf{k}(\K)$ generated by 
   $\{ t_{\mathrm{I}}\mathbf{v}^{(\sigma,\mathbf{h})}\,|\,
  \sigma\in\K,\, \text{supp}(\mathbf{h})=\sigma,\,
   \mathrm{I}_{\upalpha}(\sigma)\cup \mathrm{I} =\mathrm{L}\}$.
  By the formula~\eqref{Equ:Boundary-3},
  $d_{\upalpha}((\Lambda_{\mathbf{k}}[t_1,\cdots, t_r]\otimes \mathbf{k}(\K))^{\mathrm{L}}) \subseteq
   (\Lambda_{\mathbf{k}}[t_1,\cdots, t_r]\otimes \mathbf{k}(\K))^{\mathrm{L}}$.
  So we have a decomposition of 
  $\Lambda_{\mathbf{k}}[t_1,\cdots, t_r]\otimes \mathbf{k}(\K)$ into differential $\mathbf{k}$-modules:
  \[  \Lambda_{\mathbf{k}}[t_1,\cdots, t_r]\otimes \mathbf{k}(\K) 
   = \bigoplus_{\mathrm{L}\subseteq [r]} 
    (\Lambda_{\mathbf{k}}[t_1,\cdots, t_r]\otimes \mathbf{k}(\K))^{\mathrm{L}}.\] 
So we can define a multigraded structure
on $\mathrm{Tor}^{\mathbf{k}[\mathrm{u}_1,\cdots, \mathrm{u}_r]}(\mathbf{k}, \mathbf{k}(\K))$ by:
  \begin{equation} \label{Equ:Multi-Defi-0}
  \mathrm{Tor}^{\mathbf{k}[\mathrm{u}_1,\cdots, \mathrm{u}_r]}(\mathbf{k}, \mathbf{k}(\K)) = \bigoplus_{\mathrm{L}\subseteq [r]} \mathrm{Tor}_{*,\mathrm{L}}^{\mathbf{k}[\mathrm{u}_1,\cdots, \mathrm{u}_r]}(\mathbf{k}, \mathbf{k}(\K)), \ \text{where} \ \qquad
   \end{equation}
  \begin{equation} \label{Equ:Multi-Defi}
   \mathrm{Tor}_{q,\mathrm{L}}^{\mathbf{k}[\mathrm{u}_1,\cdots, \mathrm{u}_r]}(\mathbf{k}, \mathbf{k}(\K)) := H^{-q}\big(  (\Lambda_{\mathbf{k}}[t_1,\cdots, t_r]\otimes \mathbf{k}(\K))^{\mathrm{L}}\big), 
   \ \forall q\geq 0 . 
    \end{equation}
  
 Next, we define a $\mathbf{k}$-linear map
 \begin{align} \label{Equ:Psi-alpha}
   \psi_{\upalpha}: \Lambda_{\mathbf{k}}[t_1,\cdots, t_r]\otimes \mathbf{k}(\K) &\longrightarrow C^*(\mathcal{X}(\K, \upalpha);\mathbf{k}) \\
 t_{\mathrm{I}}\mathbf{v}^{(\sigma,\mathbf{h})} \ &\longmapsto\ \mathbf{B}^*_{(\sigma, \mathbf{h}, \mathrm{I})} \notag
 \end{align}
 Obviously, $\psi_{\upalpha}$ is a $\mathbf{k}$-linear
 isomorphism. Moreover, by comparing~\eqref{Equ:Boundary-2} with~\eqref{Equ:Boundary-3}, we see that $\psi_{\upalpha}$ is
 a cochain map and hence an isomorphism of cochain complexes. In addition, $ \psi_{\upalpha}$ clearly preserves the multigradings of $\Lambda_{\mathbf{k}}[t_1,\cdots, t_r]\otimes \mathbf{k}(\K)$ and $C^*(\mathcal{X}(\K, \upalpha);\mathbf{k})$. So we obtain
 the following lemma.\n 
 
    \begin{lem} \label{Lem:Quotient-Isom}
    For any $\mathrm{L}\subseteq [r]$, 
    $\psi_{\upalpha}: (\Lambda_{\mathbf{k}}[t_1,\cdots, t_r]\otimes \mathbf{k}(\K))^{\mathrm{L}} \rightarrow C^{*,\mathrm{L}}(\mathcal{X}(\K, \upalpha);\mathbf{k})$ induces an isomorphism in cohomology.    
    \end{lem}
    
    So for each $\mathrm{L}\subseteq [r]$,
    we have a diagram:
 \begin{equation} \label{Diagram-2}
          \xymatrix{
          (\Lambda_{\mathbf{k}}[t_1,\cdots, t_r]\otimes
        \mathbf{k}(\K))^{\mathrm{L}}
                \ar[r]^{\quad\ \, \psi_{\upalpha}} &
            C^{*,\mathrm{L}}(\mathcal{X}(\K, \upalpha);\mathbf{k})
            \ar[d]^{\iota^{*}_{\upalpha}} \\
         &  C^{*,\mathrm{L}}(\mathcal{Z}_{\K}\slash S_{\upalpha};\mathbf{k})
        } 
     \end{equation}
     where $\iota^{*}_{\upalpha}$ is the cochain map induced by
   the inclusion $\iota_{\upalpha}: \mathcal{Z}_{\K}\slash S_{\upalpha} \hookrightarrow \mathcal{X}(\K, \upalpha)$. So 
   \begin{equation} \label{Equ:Iota-star}
   \iota^{*}_{\upalpha}(\mathbf{B}^*_{(\sigma, \mathbf{h}, \mathrm{I})}) = \begin{cases}
    B^*_{(\sigma,\mathrm{I})}, & \text{if $\mathrm{I}\cap \mathrm{I}_{\upalpha}(\sigma) = \varnothing$ and $\mathbf{h}=\mathbf{1}_{\sigma}$}; \\
     0, & \text{otherwise}.
   \end{cases}
   \end{equation}
 Since $\mathcal{Z}_{\K}\slash S_{\upalpha}$ is a deformation retract of $\mathcal{X}(\K, \upalpha)$,
  $\iota^*_{\upalpha}$ induces an isomorphism in cohomology. So 
  $\iota^{*}_{\upalpha}\circ \psi_{\upalpha} : 
  (\Lambda_{\mathbf{k}}[t_1,\cdots, t_r]\otimes
        \mathbf{k}(\K))^{\mathrm{L}} \rightarrow 
        C^{*,\mathrm{L}}(\mathcal{Z}_{\K}\slash S_{\upalpha};\mathbf{k})$ induces an isomorphism in cohomology. More specifically, we have
        \begin{equation}
          \iota^{*}_{\upalpha}\circ \psi_{\upalpha}(t_{\mathrm{I}}\mathbf{v}^{(\sigma,\mathbf{h})})=
          \begin{cases}
    B^*_{(\sigma,\mathrm{I})}, & \text{if $\mathrm{I}\cap \mathrm{I}_{\upalpha}(\sigma) = \varnothing$ and $\mathbf{h}=\mathbf{1}_{\sigma}$}; \\
     0, & \text{otherwise}.
   \end{cases}
        \end{equation}
    Notice when $\mathrm{I}\cap \mathrm{I}_{\upalpha}(\sigma) = \varnothing$ and $\mathrm{I}_{\upalpha}(\sigma)\cup \mathrm{I}=\mathrm{L}$,
        $\dim  B^*_{(\sigma,\mathrm{I})} = 2|\I_{\upalpha}(\sigma)|+|\mathrm{I}| = 2|\mathrm{L}|-|\mathrm{I}|$. So for any $q\geq 0$,
        $\iota^{*}_{\upalpha}\circ \psi_{\upalpha}$
        induces an isomorphism 
\begin{equation} \label{Equ:Iso-Graded}
  H^{-q}( (\Lambda_{\mathbf{k}}[t_1,\cdots, t_r]\otimes
        \mathbf{k}(\K))^{\mathrm{L}}) \overset{\cong}{\longrightarrow} H^{2|\mathrm{L}|-q}(C^{*,\mathrm{L}}(\mathcal{Z}_{\K}\slash S_{\upalpha};\mathbf{k})).
    \end{equation}
    \n
     
   \begin{prop} \label{prop:Cohomology-Isomor}
  If $\upalpha=\{\alpha_1,\cdots, \alpha_r \}$ is a partition of $[m]$
  that is nondegenerate for $\mathcal{K}$, then for any $q\geq 0$ and $\mathrm{L}\subseteq [r]$, 
  $\mathrm{Tor}_{q,\mathrm{L}}^{\mathbf{k}[\mathrm{u}_1,\cdots, \mathrm{u}_r]}(\mathbf{k}, \mathbf{k}(\K))\cong \widetilde{H}^{|\mathrm{L}|-q-1}(\K|_{\omega^{\upalpha}_{\mathrm{L}}};\mathbf{k})
      $.  
 \end{prop}
 \begin{proof}
  By definition~\eqref{Equ:Multi-Defi}, $
  \mathrm{Tor}_{q,\mathrm{L}}^{\mathbf{k}[\mathrm{u}_1,\cdots, \mathrm{u}_r]}(\mathbf{k}, \mathbf{k}(\K)) =
   H^{-q}\big(  (\Lambda_{\mathbf{k}}[t_1,\cdots, t_r]\otimes \mathbf{k}(\K))^{\mathrm{L}}\big)$. So by~\eqref{Equ:Iso-Graded}, $
  \mathrm{Tor}_{q,\mathrm{L}}^{\mathbf{k}[\mathrm{u}_1,\cdots, \mathrm{u}_r]}(\mathbf{k}, \mathbf{k}(\K))\cong H^{2|\mathrm{L}|-q}(C^{*,\mathrm{L}}(\mathcal{Z}_{\K}\slash S_{\upalpha};\mathbf{k}))$.
 Moreover, it is shown in the proof of~\cite[Theorem 1.2]{Yu2019} 
 that 
 $$H^q(C^{*,\mathrm{L}}(\mathcal{Z}_{\K}\slash S_{\upalpha};\mathbf{k}))\cong
 \widetilde{H}^{q-|\mathrm{L}|-1}(\K|_{\omega^{\upalpha}_{\mathrm{L}}};\mathbf{k}).$$
        So we obtain
   $\mathrm{Tor}_{q,\mathrm{L}}^{\mathbf{k}[\mathrm{u}_1,\cdots, \mathrm{u}_r]}(\mathbf{k}, \mathbf{k}(\K))\cong\widetilde{H}^{|\mathrm{L}|-q-1}(\K|_{\omega^{\upalpha}_{\mathrm{L}}};\mathbf{k})
      $. We want to remind the reader that
     in~\cite{Yu2019}, the space $\mathcal{Z}_{\K}\slash S_{\upalpha}$ is denoted by $X(\mathcal{K},\lambda_{\upalpha})$ and the full subcomplex $\K|_{\omega^{\upalpha}_{\mathrm{L}}}$ is denoted by $\mathcal{K}_{\upalpha,\mathrm{L}}$.
 \end{proof}

    \n
    Finally, let us give a proof of Theorem~\ref{Thm:Main-1}. 
  Our proof will use the following theorem in~\cite{Cha91} on 
  the Betti numbers of
    any multigraded module over a polynomial ring, which
    generalizes~\cite[Corollary 2.5]{EvanGriff88}. 
   \n
    
     \begin{thm}[Theorem 3 in~\cite{Cha91}] \label{Thm:Cha}
     Let $M$ be a multigraded $\mathbf{k}[\mathrm{u}_1,\cdots,\mathrm{u}_r]$-module of Krull dimension $s$. Then
     $\dim_{\mathbf{k}} \mathrm{Tor}^{\mathbf{k}[\mathrm{u}_1,\cdots,\mathrm{u}_r]}_{i}(\mathbf{k};M)\geq \binom{r-s}{i}$ for every $i\geq 0$. 
    \end{thm}
    
     The Krull dimension $\mathrm{Kd}(R)$ of a commutative ring $R$ is the maximal number of 
      algebraically independent elements of $R$.
    The \emph{Krull dimension} $ \mathrm{Kd}_{\mathbf{k}[\mathrm{u}_1,\cdots,\mathrm{u}_r]}(M)$ of a $\mathbf{k}[\mathrm{u}_1,\cdots,\mathrm{u}_r]$-module $M$ is defined to be 
     the Krull dimension of the quotient ring of $\mathbf{k}[\mathrm{u}_1,\cdots,\mathrm{u}_r]$ which makes $M$ a faithful module. That is,
     $$ \mathrm{Kd}_{\mathbf{k}[\mathrm{u}_1,\cdots,\mathrm{u}_r]}(M) := \mathrm{Kd}\big(\mathbf{k}[\mathrm{u}_1,\cdots,\mathrm{u}_r]\slash \mathrm{Ann}(M)\big) $$
     where $\mathrm{Ann}(M)$, the \emph{annihilator},
      is the kernel of the natural map from
      $\mathbf{k}[\mathrm{u}_1,\cdots,\mathrm{u}_r]$
      to the ring of  $\mathbf{k}[\mathrm{u}_1,\cdots,\mathrm{u}_r]$-linear endomorphisms of $M$.
      \n
  Suppose
   a partition $\upalpha =\{\alpha_1,\cdots,\alpha_r\}$ 
          of $[m]$ is nondegenerate for a simplicial complex $\mathcal{K}$ on $[m]$.
     Let $\{ j_0,\cdots, j_d\}\subseteq [m]$ be a maximal simplex of $\mathcal{K}$ where $d=\dim(\mathcal{K})$ and assume that $j_l\in \alpha_{i_l}$, $0\leq l \leq d$.  
     Then with respect to the $\mathbf{k}[\mathrm{u}_1,\cdots,\mathrm{u}_r]$-module structure
     on $\mathbf{k}(\mathcal{K})$ defined in~\eqref{Equ:Mod-Struc}, $u_{j_0},\cdots,u_{j_d}$ are $d+1$ algebraically independent elements of $\mathbf{k}[\mathrm{u}_1,\cdots,\mathrm{u}_r]\slash \mathrm{Ann}(\mathbf{k}(\mathcal{K}))$. So we have
     $\mathrm{Kd}\big(\mathbf{k}[\mathrm{u}_1,\cdots,\mathrm{u}_r]\slash \mathrm{Ann}(\mathbf{k}(\mathcal{K}))\big) \geq d+1$. Conversely, it is obvious that any monomial 
     $u^{a_1}_{j_1}\cdots u^{a_s}_{j_s}$ in $\mathbf{k}[\mathrm{u}_1,\cdots,\mathrm{u}_r]$ with $s>d+1$
     acts trivially on $\mathbf{k}(\mathcal{K})$.
     So $\mathrm{Kd}\big(\mathbf{k}[\mathrm{u}_1,\cdots,\mathrm{u}_r]\slash \mathrm{Ann}(\mathbf{k}(\mathcal{K}))\big)\leq d+1$. So we obtain
    \begin{equation} \label{Equ:Krull-Dim}
      \mathrm{Kd}_{\mathbf{k}[\mathrm{u}_1,\cdots,\mathrm{u}_r]}(\mathbf{k}(\mathcal{K}))=\dim(\mathcal{K})+1.
      \end{equation}

    \nn
    
  \noindent \textbf{\textit{Proof of Theorem~\ref{Thm:Main-1}}}:\n
   
   By Proposition~\ref{prop:Cohomology-Isomor} and the Hochster's formula~\eqref{Equ:Hochster}, we have:
       \[   
        \dim_{\mathbf{k}} \mathrm{Tor}_{q,\mathrm{L}}^{\mathbf{k}[\mathrm{u}_1,\cdots, \mathrm{u}_r]}(\mathbf{k}, \mathbf{k}(\K)) = \dim_{\mathbf{k}} 
        \widetilde{H}^{|\mathrm{L}|-q-1}(\K|_{\omega^{\upalpha}_{\mathrm{L}}};\mathbf{k}) 
         = \beta^{\mathbf{k}(\mathcal{K})}_{q+|\omega^{\upalpha}_{\mathrm{L}}|-|\mathrm{L}|,
        \omega^{\upalpha}_{\mathrm{L}}}, \ \forall q\geq 0.  \]
    Then since
   the Krull dimension of $\mathbf{k}(\mathcal{K})$
   with respect to the $\mathbf{k}[\mathrm{u}_1,\cdots,\mathrm{u}_r]$-module structure in~\eqref{Equ:Mod-Struc} is equal to
   $\dim(\mathcal{K})+1$, we obtain the following inequality from Theorem~\ref{Thm:Cha}
    for any $q\geq 0$:
   \begin{align*}
     \sum_{\mathrm{L}\subseteq [r]} \beta^{\mathbf{k}(\mathcal{K})}_{q+|\omega^{\upalpha}_{\mathrm{L}}|-|\mathrm{L}|,\omega^{\upalpha}_{\mathrm{L}}} &=
        \sum_{\mathrm{L}\subseteq [r]} \dim_{\mathbf{k}} \mathrm{Tor}_{q,\mathrm{L}}^{\mathbf{k}[\mathrm{u}_1,\cdots, \mathrm{u}_r]}(\mathbf{k}, \mathbf{k}(\K))  \\
        & \overset{\eqref{Equ:Multi-Defi-0}}{=}  \dim_{\mathbf{k}}\mathrm{Tor}_{q}^{\mathbf{k}[\mathrm{u}_1,\cdots, \mathrm{u}_r]}(\mathbf{k}, \mathbf{k}(\K)) \geq \binom{r-\dim(\mathcal{K})-1}{q}.
      \end{align*}  
    Then we have
    \[ \sum_{q\geq 0} \sum_{\mathrm{L}\subseteq [r]} \dim_{\mathbf{k}} 
        \widetilde{H}^{|\mathrm{L}|-q-1}(\K|_{\omega^{\upalpha}_{\mathrm{L}}};\mathbf{k}) = \sum_{q\geq 0} \sum_{\mathrm{L}\subseteq [r]} \beta^{\mathbf{k}(\mathcal{K})}_{q+|\omega^{\upalpha}_{\mathrm{L}}|-|\mathrm{L}|,\omega^{\upalpha}_{\mathrm{L}}} \geq 2^{r-\dim(\mathcal{K})-1}.\]
        This finishes the proof of the theorem.
    \qed
    
    \begin{rem}
     It is not clear whether 
     Theorem~\ref{Thm:Main-1} should hold for all partitions of $[m]$. Indeed, if a partition $\upalpha$ of $[m]$ is not nondegenerate for $\mathcal{K}$, 
      there is no natural way to identify
      $ H^*(\mathcal{Z}_{\mathcal{K}}\slash S_{\upalpha};\mathbf{k}) $ with the Tor algebra of any 
    multigraded
     $\mathbf{k}[\mathrm{u}_1,\cdots, \mathrm{u}_r]$-module. However, we have not found any counterexample
     to the statement of
      Theorem~\ref{Thm:Main-1} among general 
      partitions of $[m]$ either.
    \end{rem}

\vskip .4cm  

 \section*{Acknowledgements}
  The author is partially supported by
 Natural Science Foundation of China (grant No.11871266) and 
 the PAPD (Priority Academic Program Development) of Jiangsu Higher Education Institutions.

\end{document}